\documentclass[a4paper,11pt]{amsart}

\usepackage{amsfonts,amssymb,mathtools}
\usepackage{tikz}
\usepackage{textcomp, color}
\usepackage{hyperref}

\theoremstyle{definition}

\newtheorem{definition}{Definition}[section]

\theoremstyle{plain}

\newtheorem*{thmA}{Theorem A}
\newtheorem*{thmB}{Theorem B}
\newtheorem{theorem}[definition]{Theorem}
\newtheorem{lemma}[definition]{Lemma}
\newtheorem{proposition}[definition]{Proposition}

\newcommand{\F}{\mathbb{F}}

\newcommand{\N}{\mathbb{N}}

\newcommand{\maxG}{\max\nolimits_{G}\hspace{.09 em}}

\begin{document}

\title[Commutators in finite $p$-groups]{Commutators in finite $p$-groups with $2$-generator derived subgroup} 

\author[I.\ de las Heras]{Iker de las Heras}
\address{Department of Mathematics\\ University of the Basque Country UPV/EHU\\
48080 Bilbao, Spain}
\email{iker.delasheras@ehu.eus}

\author[G.A.\ Fern\'andez-Alcober]{Gustavo A.\ Fern\'andez-Alcober}
\address{Department of Mathematics\\ University of the Basque Country UPV/EHU\\
48080 Bilbao, Spain}
\email{gustavo.fernandez@ehu.eus}

\thanks{Both authors are supported by the Spanish Government grant MTM2017-86802-P and by the Basque Government grant IT974-16.
The first author is also supported by a predoctoral grant of the University of the Basque Country, and the second author, by the Spanish Government grant MTM2014-53810-C2-2-P}

\begin{abstract}
Let  $G$ be a finite $p$-group whose derived subgroup $G'$ can be generated by $2$ elements.
If $G'$ is abelian, Guralnick proved that every element of $G'$ is a commutator.
In this paper, we prove that the condition that $G'$ should be abelian is not needed.
Even more, we prove that every element of $G'$ is a commutator of the form $[x,g]$ for a fixed $x\in G$.
\end{abstract}

\maketitle

\section{Introduction}

Soon after the introduction of commutators in group theory on the eve of the 20th century, it was observed that the set $K(G)$ of commutators of a group $G$ need not be a subgroup.
In other words, the derived subgroup $G'=\langle K(G) \rangle$ may be strictly larger than $K(G)$.

However, several families of groups are known for which the equality $G'=K(G)$ holds.
A remarkable result in this direction is the proof by Liebeck, O'Brien, Shalev, and Tiep \cite{LOST} in 2010 of the so-called Ore Conjecture, according to which every finite simple group $G$ satisfies the condition $G'=K(G)$.
This is also true, at the other end of the spectrum, for nilpotent groups with cyclic derived subgroup \cite{rod}, a result that fails to hold if we drop the nilpotency assumption \cite{mac}.

The study of this property for finite nilpotent groups is obviously reduced to finite $p$-groups, where $p$ is a prime.
In this case, Guralnick \cite[Theorem 3.1]{gur} showed that $G'=K(G)$ whenever $G'$ is abelian and can be generated by $2$ elements.
Guralnick himself \cite[Theorem B]{gur} extended this result to 3-generator abelian derived subgroups provided that $p>3$, and showed that it fails to hold for $p=2$ or $3$ (see Examples 3.5 and 3.6 in the same paper).
On the other hand, as shown in Macdonald's book \cite{mac2} (Exercise 5, page 78), for every prime $p$ there exists a finite $p$-group with a $4$-generator abelian derived subgroup and such that
$G'\ne K(G)$.
A wealth of information about the condition $G'=K(G)$ can be found in the papers \cite{KM1} and \cite{KM2}.

Our goal in this paper is to generalise the above-mentioned first result of Guralnick by showing that the requirement that $G'$ should be abelian is not necessary.
Actually, we obtain the stronger property that all elements of $G'$ arise as commutators from a single suitable element of $G$.
This result is stated without proof in \cite{gur} in the case that $G'$ is abelian.
Thus the main theorem of this paper is the following.

\begin{thmA}
Let $G$ be a finite $p$-group.
If $G'$ can be generated by $2$ elements, then $G'=\{[x,g] \mid g\in G\}$ for a suitable $x\in G$.
\end{thmA}

Our proof of Theorem A does not need Guralnick's result for $p$-groups with abelian derived subgroup, and not even the result about $p$-groups with cyclic derived subgroup, which is actually an easy consequence of our methods.
Our approach relies on observing that $G'$ is necessarily powerful if $G$ is a finite $p$-group and
$d(G')\le 2$.
Recall that a finite $p$-group $P$ is said to be powerful if $P'\le P^{\hspace{0.4pt}p}$ if $p$ is odd, or if $P'\le P^4$ if $p=2$.

\vspace{8pt}

As an immediate consequence of Theorem A, we get the result below about pro-$p$ groups with $2$-generator derived subgroup.
Observe that, if $G$ is a pro-$p$ group and $\overline{G'}$ is topologically finitely generated, then $G'$ is closed in $G$ and consequently we can drop the closure operation from $\overline{G'}$.
This follows by passing to a suitable finitely generated closed subgroup of $G$ and using Proposition 1.19 of \cite{DDMS}.

\begin{thmB}
Let $G$ be a pro-$p$ group.
If $G'$ can be topologically generated by $2$ elements, then $G'=\{[x,g] \mid g\in G\}$ for a suitable
$x\in G$.
\end{thmB}

\vspace{8pt}

\textit{Notation.\/}
Let $G$ be a group, and let $H\le G$ and $N\trianglelefteq G$.
We write $H\max G$ to denote that $H$ is maximal in $G$, and $H\maxG N$ to denote that $H$ is maximal among the proper subgroups of $N$ that are normal in $G$.
If $x\in G$ then we set $K_x(H)=\{ [x,h]\mid h\in H \}$ and $[x,H]=\langle K_x(H) \rangle$.
If $G$ is finitely generated, $d(G)$ stands for the minimum number of generators of $G$.
Finally, if $G$ is a topological group, we write $N\trianglelefteq_{\mathrm{o}} G$ to denote that $N$ is an open normal subgroup of $G$.

\section{Proof of the main results}

We start by recalling the following result of Blackburn \cite[Theorem 1]{black}, which will allow us to reduce the proof of Theorem A to the case when $G'$ is powerful.

\begin{lemma}
\label{blackburn}
Let $G$ be a finite $p$-group such that $d(G')\le 2$.
Then either $G'$ is abelian or it can be generated by two elements $a$ and $b$ with defining relations $a^{p^m}=b^{p^{n+k}}=1$ and $[a,b]=b^{p^n}$, with $k>0$ and $n\ge m\ge 2k$.
\end{lemma}

As a consequence, if the derived subgroup of a finite $p$-group $G$ can be generated by two elements, we have $G''\le (G')^{p^2}$ and $G'$ is powerful.

\vspace{10pt}

We continue by proving a couple of easy results about powerful finite $p$-groups.
Throughout the paper, we use freely the basic theory of powerful groups, as developed in
\cite[Chapter 2]{DDMS} or \cite[Chapter 11]{khu}.
Observe that the Frattini subgroup of a powerful $p$-group $G$ coincides with $G^p$, and consequently $|G:G^p|=p^{d(G)}$.

\begin{lemma}
\label{lemma powerful}
Let $G$ be a powerful finite $p$-group.
Then:
\begin{enumerate}
\item
If $H\le G$ is also powerful then $|G^{p^i}:H^{p^i}|\le |G:H|$ for all $i\ge 0$.
\vspace{2pt}
\item
If $d(G)=2$ then every subgroup of $G$ is also powerful.
\end{enumerate}
\end{lemma}

\begin{proof}
(i)
Since $G$ is powerful, we have $d(H)\le d(G)$ by Theorem 11.18 of \cite{khu}.
Since $H$ is also powerful, this amounts to $|H:H^p|\le |G:G^p|$, which yields the result for $i=1$.
Now $G^p$ and $H^p$ are again powerful, and we have $G^{p^i}=(G^p)^{p^{i-1}}$ and similarly for $H$, so the general case follows immediately by induction on $i$.

(ii)
By induction on the group order, it suffices to show that every maximal subgroup $M$ of $G$ is powerful.
Now since $G$ is powerful and $d(G)=2$, we have $|M:G^p|=p$.
Since $G^p$ is powerfully embedded in $G$, it follows from \cite[Lemma 11.7]{khu} that $M$ is powerful.
\end{proof}

Next we see how to extend the covering of a subgroup with commutators of a fixed element $x$ from
a factor group to the group.

\begin{lemma}
\label{lemma index 2}
Let $G$ be a group and let $N\le L\le G$, with $N$ normal in $G$.
Suppose that for some $x\in G$ the following two conditions hold:
\begin{enumerate}
\item
$L/N\subseteq K_{xN}(G/N)$.
\item
$N\subseteq K_x(G)$.
\end{enumerate}
Then $L\subseteq K_x(G)$.
\end{lemma}

\begin{proof}
We prove that every coset $yN$ with $y\in L$ lies in $K_x(G)$.
By using (i) and (ii), for some $g\in G$ we have
\[
yN = [x,g]N = [x,g]N^g \subseteq [x,g] K_x(G)^g.
\]
Now since every element of $[x,g] K_x(G)^g$ is of the form $[x,g][x,h]^g=[x,hg]$, the result follows.
\end{proof}

The previous lemma will be used in combination with the following result, whose proof is straightforward.

\begin{lemma}
\label{lemma central}
Let $G$ be a group and let $N\le L\le G$, with $N$ normal in $G$.
If $L/N=\langle [x,s]N \mid s\in S \rangle$ for some $x\in G$ and some $S\subseteq G$ with $[L,S]\subseteq N$,
then $L/N\subseteq K_{xN}(\langle S \rangle N/N)\subseteq K_{xN}(G/N)$.
\end{lemma}

Next we give a lemma which is the key to our proof of Theorem A.
It shows that, under some specific conditions, covering a factor group $L/L^p$ with commutators of a given element $x$ is enough to cover $L$.
We recall that, if $L$ and $N$ are two normal subgroups of a group $G$ and $n\in\N$, then
 $[L^n,N]\le [L,N]^n [L,N,L]$.

\begin{lemma}
\label{lemma LN}
Let $G$ be a finite $p$-group and let $N\le L$ be normal subgroups of $G$, with $L$ powerful and
$d(L)\le 2$.
Then the following hold:
\begin{enumerate}
\item
If there exist $x,g\in G$ such that $L/N=\langle [x,g]N\rangle$ and $[x,g,g]\in N^p$, then
$L^{p^i}/N^{p^i}=\langle [x,g^{p^i}]N^{p^i}\rangle$
for every $i\ge 1$. 
\item
Assume furthermore that $L^p\le N$ and $|L:N|=p$.
If there exist $x,g,h\in G$ such that $L/N=\langle [x,g]N \rangle$ and $N/L^p=\langle [x,h]L^p \rangle$ with
$[x,g,g]\in N^p$ and $[x,h,h]\in L^{p^2}$, then $L\subseteq K_x(G)$.
\end{enumerate}
\end{lemma}

\begin{proof}
(i)
We argue by induction on $i$.
Assume first $i=1$.
Since $L$ is powerful and $L=\langle [x,g], N \rangle$, it follows that $L^p=\langle [x,g]^p, N^p \rangle$, and thus $L^p/N^p=\langle [x,g]^pN^p\rangle$.
Now, from the hypothesis $[x,g,g]\in N^p$ we get
\begin{equation}
\label{eqn comm [x,gp]}
[x,g^p]\equiv [x,g]^p \pmod{N^p},
\end{equation}
and consequently $L^p/N^p=\langle [x,g^p]N^p\rangle$.

Now let $i>1$.
By (ii) of Lemma \ref{lemma powerful}, $N$ is also powerful.
If we prove that $[x,g^p,g^p]\in N^{p^2}$, then we can apply the induction hypothesis with $L^p$ and $N^p$ playing the role of $L$ and $N$, and $g^p$ playing the role of $g$, and we are done.
Observe that $|N^p:N^{p^2}|\le p^2$, since $N^p$ is powerful and $d(N^p)\le d(L)\le 2$.
Since $N$ is normal in $G$, it follows that $[N^p,G']\le [N^p,G,G]\le N^{p^2}$.
As a consequence, $[N^p,G^p]\le [N^p,G]^p[N^p,G,G]\le N^{p^2}$ and then (\ref{eqn comm [x,gp]}) yields that
\[
[x,g^p,g^p] \equiv [[x,g]^p,g^p] \pmod{N^{p^2}}.
\]
On the other hand, since $[x,g,g]\in N^p$ implies that $[x,g,g^p]\in N^p$ as well, it follows that 
$[[x,g],g^p,[x,g]]\in [N^p,G']\le N^{p^2}$, and we obtain as desired that
\[
[x,g^p,g^p] \equiv [[x,g]^p,g^p] \equiv [x,g,g^p]^p \equiv 1 \pmod{N^{p^2}}.
\]

(ii)
Consider the following normal series of $G$:
\begin{equation}
\label{series}
L\ge N\ge L^p\ge N^p\ge L^{p^2}\ge N^{p^2}\ge \cdots \ge 1.
\end{equation}
By hypothesis, we have $|L:N|=p$.
Also, since $L$ is powerful and $d(L)\le 2$, we have $|L:L^p|\le p^2$, and therefore $|N:L^p|\le p$.
As a consequence, if $R$ and $S$ are two consecutive terms of (\ref{series}) then $|R:S|\le p$, by using (i) of Lemma \ref{lemma powerful}.
Hence the section $R/S$ is central in $G$.
On the other hand, by (i), $R/S=\langle [x,y]S \rangle$ for some $y\in G$.
Then $R/S\subseteq K_{xS}(G/S)$, by Lemma \ref{lemma central}, and by going up the series (\ref{series}) and using Lemma \ref{lemma index 2}, we conclude that $L\subseteq K_x(G)$.
\end{proof}

As an easy illustration of our method based on Lemma \ref{lemma LN}, let us prove the well-known result that if $G$ is a finite $p$-group and $G'$ is cyclic then $G'$ consists of commutators.
Let $C^*=C_G(G'/(G')^{p^2})$.
Then $|G:C^*|\le p$ and $G'=[G,C^*]$.
By the Burnside Basis Theorem, we have $G'=\langle [x,y] \rangle$ for some $x\in G$ and $y\in C^*$.
Then $[x,y,y]\in [G,C^*,C^*]\le [G',C^*]\le (G')^{p^2}$ and we can apply Lemma \ref{lemma LN} with $L=G'$ and $N=(G')^p$ to get $G'=K_x(G)$.

\vspace{8pt}

However, if $d(G')=2$ the situation is more complicated and actually it is not always possible to apply Lemma \ref{lemma LN} with $L=G'$, being necessary to start the `domino effect' shown in the last proof at a lower level.
Also, we need to cover two chief factors, instead of one, with commutators of a single element $x$ before taking $p$th powers, and this is trickier in some cases when dealing with the prime $2$.
This motivates introducing the subgroups $D(T)$ and $R(U)$ of Definitions \ref{definition D} and \ref{definition R} below which, as shown in Lemmas \ref{lemma D} and \ref{lemma R}, satisfy that any element $x$ avoiding them will have the desired covering property.

\vspace{8pt}

The following subgroup plays a fundamental role in Guralnick's proof in the case that $G'$ is abelian, and it is also relevant in our proof.

\begin{definition}
Let $G$ be a finite $p$-group with $G'$ powerful.
We define $C=C_G(G'/(G')^p)$.
\end{definition}

If $p$ is odd then, by definition, $G'$ is powerfully embedded in $C$, and so all power subgroups $(G')^{p^i}$ are also powerfully embedded in $C$.
In other words, we have $[(G')^{p^i},C]\le (G')^{p^{i+1}}$.
As it turns out, this is true for all primes, and similar inclusions hold for other commutator subgroups involving $C$.
More precisely, we have the following lemma, that we state in a bit more generality than we actually need.

\begin{lemma}
\label{lemma C}
Let $G$ be a finite $p$-group with $G'$ powerful.
Then:
\begin{enumerate}
\item
$[(G')^{p^i},C^{p^j}]\le (G')^{p^{i+j+1}}$ for every $i,j\ge 0$.
\item
$[G,C^{p^j}]\le (G')^{p^j}$ for every $j\ge 0$.
\item
If $d(G')\le 2$ then $|G:C|\le p$.
\end{enumerate}
\end{lemma}

\begin{proof}
(i)
We use induction on $j$.
Assume first that $j=0$ and use induction on $i$.
The base of the induction is given by the definition of $C$, and if $i>0$ then
\begin{align*}
[(G')^{p^i},C]
&\le
[(G')^{p^{i-1}},C]^p \ [(G')^{p^{i-1}},C,(G')^{p^{i-1}}]
\\
&\le
(G')^{p^{i+1}} [(G')^{p^i},G']
\le (G')^{p^{i+1}},
\end{align*}
by using the induction hypothesis and the fact that $G'$ is powerful.
Now if $j>0$ then
\begin{align*}
[(G')^{p^i},C^{p^j}]
&\le [(G')^{p^i},C^{p^{j-1}}]^p \ [(G')^{p^i},C^{p^{j-1}},C^{p^{j-1}}]
\\
&\le (G')^{p^{i+j+1}} [(G')^{p^{i+j}},C]
\le (G')^{p^{i+j+1}}.
\end{align*}

(ii)
Again we argue by induction on $j$, the case $j=0$ being obvious.
If $j>0$ then
\[
[G,C^{p^j}]
\le
[G,C^{p^{j-1}}]^p \ [G,C^{p^{j-1}},C^{p^{j-1}}]
\le
(G')^{p^j} \ [(G')^{p^{j-1}},C] \le (G')^{p^j},
\]
where the last inclusion follows from (i).

(iii)
This is obvious, since $C$ is the centraliser of the normal section $G'/(G')^p$, of order at most $p^2$.
\end{proof}

\begin{definition}
\label{definition D}
Let $G$ be a non-abelian finite $p$-group.
For every $T\maxG G'$ we define the subgroup $D(T)$ by the condition
\[
D(T)/T = Z(G/T),
\]
that is, $D(T)$ is the largest subgroup of $G$ satisfying $[D(T),G]\le T$.
We set $D=\cup \{ D(T) \mid T\maxG G' \}$.
\end{definition}

\begin{lemma}
\label{lemma D}
If $G$ is a non-abelian finite $p$-group then $[x,G]=G'$ if and only  if $x\not\in D$.
Furthermore, if $d(G')\le 2$ then:
\begin{enumerate}
\item
For every $T\maxG G'$, we have $\Phi(G)\le D(T)\le C$ and $\log_p |G:D(T)|$ is even.
\item
$D$ is a proper subset of $G$.
\end{enumerate}
\end{lemma}

\begin{proof}
Since $[x,G]$ is a normal subgroup of $G$, we have $[x,G]<G'$ if and only if $x\in D(T)$ for some $T\maxG G'$, and the first assertion follows.
Let us now assume that $d(G')\le 2$, and recall that $G'$ is powerful by Lemma \ref{blackburn}.
Thus $|G':(G')^p|\le p^2$.

(i)
Let $T\maxG G'$.
We have $[\Phi(G),G]=[G^p,G]\gamma_3(G)=(G')^p\gamma_3(G)\le T$, and so $\Phi(G)\le D(T)$.
Thus $G/D(T)$ can be seen as an $\F_p$-vector space.
Also $[D(T),G']\le [D(T),G,G]\le [T,G]\le (G')^p$, since $|T:(G')^p|\le p$, and consequently $D(T)\le C$.
On the other hand, the  commutator map in $G/T$ induces a a non-degenerate alternating form on $G/D(T)$, and thus
$\dim_{\F_p} G/D(T)$ is even.

(ii)
Assume for a contradiction that $D=G$.
If $|G':(G')^p|=p$ then $D=D((G')^p)$ and consequently $G'=[D,G]\le (G')^p$, a contradiction.
Thus $|G':(G')^p|=p^2$.
Let $x\in G$ be arbitrary.
Then $x\in D(T)$ for some $T\maxG G'$ and the image of $[x,G]$ in $\overline G=G/(G')^p$ has order at most $p$.
It follows that all conjugacy class lengths in $\overline G$ are either $1$ or $p$, i.e.\ that $\overline G$ is a $p$-group of breadth at most $1$.
By Lemma 2.12 of \cite{ber}, this implies that $|\overline G'|\le p$, which is again a contradiction.
\end{proof}

By the previous lemma, if $d(G')\le 2$ then there always exists $x\in G$ such that $G'=[x,G]$.
Since $|G:C|\le p$ by Lemma \ref{lemma C} and $D\subseteq C$ by Lemma \ref{lemma D}, we can  choose $x\not\in D$ such that $G=\langle x \rangle C$, and then we get $G'=[x,C]$.
We are now in a position to prove Theorem A for $p>2$.

\begin{theorem}
\label{theorem p odd}
Let $G$ be a finite $p$-group, where $p$ is an odd prime, and assume that $d(G')=2$.
Then $G'=K_x(G)$ for a suitable $x\in G$.
\end{theorem}

\begin{proof}
We have $G'=[x,C]$ for some $x\in G$ and so, by Lemma \ref{lemma central},
$G'/(G')^p=\{[x,u](G')^p \mid u\in C\}$ .
By Lemma \ref{lemma index 2}, we only need to prove that $(G')^p\subseteq K_x(G)$.
Hence we may assume that $(G')^p\ne 1$.

Since $G'$ is powerful by Lemma \ref{blackburn}, the map $g(G')^p\longmapsto g^p(G')^{p^2}$ is an epimorphism from $G'/(G')^p$ to $(G')^p/(G')^{p^2}$.
Thus $(G')^p/(G')^{p^2}$ consists of the cosets $[x,u]^p(G')^{p^2}$ with $u\in C$.
Now if $u\in C$ then, by the Hall-Petresco identity, $[x,u^p]=[x,u]^p w$ for some $w\in (H')^p\gamma_p(H)$, where
$H=\langle (u^{-1})^x,u \rangle=\langle u, [x,u] \rangle$.
Then $H'\le [G,C,C]\le (G')^p$ and $(H')^p\le (G')^{p^2}$, and since $p$ is odd, also
$\gamma_p(H)\le [(G')^p,C]\le (G')^{p^2}$ by Lemma \ref{lemma C}.
Hence $[x,u]^p\equiv [x,u^p] \pmod{(G')^{p^2}}$ for every $u\in C$.
It follows that every element of $(G')^p$ is of the form $[x,u^p]$ modulo $(G')^{p^2}$ for some $u\in C$.

Now let us choose a subgroup $T$ between $(G')^p$ and $(G')^{p^2}$ with $|(G')^p:T|=p$.
Thus both $(G')^p/T$ and $T/(G')^{p^2}$ are cyclic, generated by the image of some commutator
$[x,u^p]$ with $u\in C$.
By Lemma \ref{lemma C}, we have $[x,u^p,u^p]\in [G,C^p,C^p]\le [(G')^p,C^p]\le (G')^{p^3}$.
Thus we can apply (ii) of Lemma \ref{lemma LN} with $L=(G')^p$ and $N=T$ to get $(G')^p\subseteq K_x(G)$, as desired.
\end{proof}

Now we are concerned with the proof of Theorem A for finite $2$-groups, which is quite more involved.
The main difficulty arises when $C=G$, and in order to deal with that case, we introduce the following subgroups.

\begin{definition}
\label{definition R}
Let $G$ be a finite $2$-group such that $(G')^2\ne 1$.
For every $U\maxG (G')^2$ we define the subgroup $R(U)$ by the condition
\[
R(U)/U = C_{G/U}(G^2/U).
\]
In other words, $R(U)$ is the largest subgroup of $G$ satisfying $[R(U),G^2]\le U$.
We set $R=\cup \{ R(U) \mid U\maxG (G')^2 \}$.
\end{definition}

\begin{lemma}
\label{lemma R}
Let $G$ be a finite $2$-group with $d(G')\le 2$.
Assume furthermore that $C=G$ and that $(G')^2\ne 1$.
Then the following hold:
\begin{enumerate}
\item
$[G,G^2]=(G')^2$.
\item
$[x,G^2]=(G')^2$ if and only if $x\not\in R$.
\item
$G^2\le R(U)<G$ for every $U\maxG (G')^2$.
\item
$R(U)\cap R(V)\le R(W)$ for every $U,V,W\maxG (G')^2$ with $U\ne V$.
\end{enumerate}
\end{lemma}

\begin{proof}
(i)
The subgroups $[G,G^2]$ and $(G')^2$ coincide modulo $\gamma_3(G)$.
Since groups of exponent $2$ are abelian, we have $G'\le G^2$, which implies that
$\gamma_3(G)\le [G,G^2]$.
On the other hand, $C=G$ implies that $\gamma_3(G)\le (G')^2$.
We conclude that $[G,G^2]=(G')^2$.

(ii)
If $[x,G^2]=(G')^2$ then obviously $x\not\in R$.
On the other hand, if $[x,G^2]\ne (G')^2$ then $[x,G^2]<(G')^2$ by (i).
Let $N=[x,G^2](G')^4$.
Then $N$ is a proper subgroup of $(G')^2$, and normal in $G$, since $[(G')^2,G]\le (G')^4$ by (i)
of Lemma \ref{lemma C}.
If we consider $U\maxG (G')^2$ containing $N$, then $x\in R(U)\subseteq R$.
This proves the result.

(iii)
By (i), $R(U)$ is a proper subgroup of $G$.
On the other hand, we have
\[
[G^2,G^2] \le [G,G^2]^2 [G,G^2,G] = (G')^4 [(G')^2,G] = (G')^4.
\]
Since $(G')^4\le U$, it follows that $G^2\le R(U)$. 

(iv)
This is clear from the definition of $R(U)$, since $d((G')^2)\le 2$ implies $U\cap V=(G')^4\le W$.
\end{proof}

The next result will allow us to complete easily the proof of Theorem A in the case when $p=2$ and $C=G$.
Its proof is long and technical, and it requires a careful analysis of the relative positions of the subgroups $D(T)$ and $R(U)$, where $T\maxG G'$ and $U\maxG (G')^2$.

\begin{proposition}
\label{proposition D cup R not G}
Let $G$ be a finite $2$-group with $d(G')=2$ and $C=G$.
Then there exists $x\in G$ such that $G'=[x,G]$ and $(G')^2=[x,G^2]$.
\end{proposition}

\begin{proof}
We know that $G'=[a,G]=\langle [a,b], [a,c] \rangle$ for some $a,b,c\in G$.
If we set $H=\langle a,b,c \rangle$ then $H'=G'$, and the result immediately follows for $G$ once it is proved for $H$, taking into account that $[G^2,G]=(G')^2$ by (i) of Lemma \ref{lemma R}.
Thus we may assume that $d(G)\le 3$ and, obviously, also that $(G')^2\ne 1$.

In the remainder of the proof, let $Z/(G')^2$ be the centre of $G/(G')^2$.
Observe that $|G:Z|>4$, since otherwise the derived subgroup of $G/(G')^2$ is cyclic, and consequently $G'$ is cyclic.
Since, again by (i) of Lemma \ref{lemma R}, we have $\Phi(G)=G^2\le Z$, it follows that $|G:G^2|=8$ and that $Z=G^2$.

By Lemmas \ref{lemma D} and \ref{lemma R}, it suffices to show that $D\cup R$ does not cover the whole of $G$, since then any $x\in G\smallsetminus (D\cup R)$ satisfies both $G'=[x,G]$ and $(G')^2=[x,G^2]$.
Since $G^2\le D(T),R(U)$ for all $T\maxG G'$ and $U\maxG (G')^2$, we can prove the non-covering property by working in the group $G/G^2$ of order $8$.
Thus, if we use the bar notation in this factor group, then we need to prove that
$|\overline D\cup \overline R|\le 7$.
We do this by first determining the order of $\overline D$ and then analysing the position of the subgroups $R(U)$ with respect to $D$ and among themselves.
Before proceeding, observe that the sections $G'/(G')^2$ and $(G')^2/(G')^4$ are central in $G$ by
Lemma \ref{lemma C}, since $C=G$.
Hence the conditions $T\maxG G'$ and $U\maxG (G')^2$ are equivalent in this case to $T\max G'$ and
$U\max (G')^2$, respectively.

We claim that $|\overline D|=4$ and that $D$ is a maximal subgroup of $G$.
Let us consider an arbitrary $T\maxG G'$, and observe that there are three choices for $T$, since $d(G')=2$.
First of all, note that $|\overline{D(T)}|=2$, since $\log_2 |G:D(T)|$ is even and $D(T)$ is proper in $G$ by Lemma \ref{lemma D}.
Thus $D(T)'=[D(T),G^2]\le (G')^2$.
Now let $S\maxG G'$ with $S\ne T$.
Then $[D(S),D(T)]\le S\cap T\le (G')^2$, and as a consequence $\langle D \rangle'\le (G')^2$.
Also, if $D(S)=D(T)$  then $[D(T),G]\le (G')^2$ and $D(T)\le Z=G^2$.
Hence $|\overline{D(T)}|=1$, which is a contradiction.
Thus $\overline D$ is the union of three different subgroups of order $2$, and $|\overline D|=4$.
Since $D\subseteq \langle D \rangle \le G$ and $\langle D \rangle'\ne G'$, it follows that $D$ is a maximal subgroup of $G$, as claimed.

Now we start the analysis of the position of the subgroups of the form $R(U)$.
Since $G'$ is a $2$-generator powerful group, we have $|(G')^2:(G')^4|\le 4$.
Hence $(G')^2$ has at most 3 maximal subgroups, and the intersection of two different maximal subgroups is $(G')^4$.
By (iii) of Lemma \ref{lemma R}, all the $R(U)$ are proper in $G$, and if none of them is maximal in $G$, we immediately get $|\overline D\cup \overline R|\le 7$.
The same conclusion holds if $R(U)=D$ whenever $R(U)\max G$.
Thus we may assume that there exists $U\maxG (G')^2$ such that $R(U)\max G$,
i.e.\ such that $|\overline{R(U)}|=4$, and furthermore $R(U)\ne D$.

Then $|\overline D \cup \overline{R(U)}|=6$, and we may also assume that there exists another
$V\maxG (G')^2$ such that $R(V)\not\subseteq D\cup R(U)$.
This implies in particular that $d((G')^2)=2$, and $(G')^2$ has exactly 3 maximal subgroups.
Also, since $G'$ is powerful, the square map induces an isomorphism between $G'/(G')^2$ and
$(G')^2/(G')^4$.
As a consequence,
\begin{equation}
\label{eqn square injective}
g\in G'\smallsetminus (G')^2
\
\Longrightarrow
\
g^2\in (G')^2\smallsetminus (G')^4,
\end{equation}
and also all three maximal subgroups of $(G')^2$ are of the form $T^2$, where $T\maxG G'$.

Let $W$ be the third maximal subgroup of $(G')^2$, apart from $U$ and $V$.
If $\overline{R(W)}=\overline 1$ then, since $R(U)\cap R(V)\le R(W)$ by (iv) of Lemma \ref{lemma R}, it follows that $|\overline{R(V)}|\le 2$ and consequently $|\overline D\cup \overline R|\le 7$.
Hence we may assume that $|\overline{R(W)}|\ge 2$.

Now we consider two separate cases:

\vspace{5pt}

\textit{Case 1: $R(W)\le R(U)$.}

\vspace{5pt}

Again by (iv) of Lemma \ref{lemma R}, we get $R(W)=R(U)\cap R(W)<R(V)$, with proper inclusion since
$R(V)\not\le R(U)$.
In particular, $|\overline{R(W)}|=2$ and $|\overline{R(V)}|=4$, which implies that $|\overline R|=6$.

Assume first that $D\cap R(U)\ne D\cap R(V)$.
In this case we have $|\overline D\cap \overline R|\ge 3$ and hence $|\overline D\cup \overline R|\le 7$, as desired.
\begin{figure}[h]
\centering
\begin{tikzpicture}[every node/.style={fill=white}] 
	\draw (0,5) -- (-2,4) -- (-2,3) -- (0,2);
	\draw (-2,4) -- (0,3) -- (0,2);	
	\draw (0,5) -- (0,4) -- (-2,3);
	\draw (0,4) -- (2,3);
	\draw (0,5) -- (2,4) -- (0,3);
	\draw (2,4) -- (2,3) -- (0,2);

	\node at (0,5) {$G$};
	\node at (-2,4) {$D$};
	\node at (0,4) {$R(U)$};
	\node at (2,4) {$R(V)$};
	\node at (-2,3) {$D\cap R(U)$};
	\node at (0,3) {$D\cap R(V)$};
	\node at (2,3) {$R(W)$};
	\node at (0,2) {$G^2$};
\end{tikzpicture}
\caption{The case $D\cap R(U)\ne D\cap R(V)$.}
\end{figure}

Suppose now that $D\cap R(U)=D\cap R(V)$, so that this intersection coincides with $R(W)$.
Now, by the fourth paragraph of the proof, $\overline D$ has three subgroups of order $2$, which are all of the form $\overline{D(T)}$.
Thus $R(W)=D(T)$ for some $T\maxG G'$.
\begin{figure}[h]
\centering
\begin{tikzpicture}[every node/.style={fill=white}] 
	\draw (0,5) -- (0,4) -- (0,3) -- (0,2);
	\draw (0,5) -- (-2,4) -- (0,3);	
	\draw (0,5) -- (2,4) -- (0,3);

	\node at (0,5) {$G$};
	\node at (-2,4) {$D$};
	\node at (0,4) {$R(U)$};
	\node at (2,4) {$R(V)$};
	\node at (0,3) {$R(W)=D(T)$};
	\node at (0,2) {$G^2$};
\end{tikzpicture}
\caption{The case $D\cap R(U)=D\cap R(V)$.}
\end{figure}
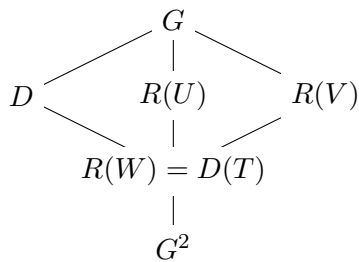

Choose $g\in R(W)\smallsetminus G^2$.
Then $[g,G]$ is contained in $T$, but since $Z=G^2$, it is not contained in $(G')^2$.
Since $G=\langle R(U), D \rangle$ and $[g,D]\le D'\le (G')^2$, there exists $h\in R(U)$ such that
$[g,h]\in T\smallsetminus (G')^2$.
Now, we have
\[
[g,h^2]\in [R(W),G^2]\le [R(U),G^2]\cap [R(V),G^2]\le U\cap V\le (G')^4,
\]
and, on the other hand,
\[
[g,h^2] = [g,h]^2 [g,h,h],
\]
where $[g,h]^2\in T^2\smallsetminus (G')^4$ by (\ref{eqn square injective}),
and $[g,h,h]\in [h,G']\le [h,G^2]\le U$.
Thus necessarily $U=T^2$.
Since the same argument can be applied with $V$ in the place of $U$, we deduce that $U=V$, which is a contradiction.

\vspace{5pt}

\textit{Case 2: $R(W)\not\le R(U)$.}

\vspace{5pt}

We are going to prove that this case is impossible.
Choose an element $x\in R(V)\smallsetminus (D\cup R(U))$.
Then $G=R(U) \cup xR(U)$.
Since $R(W)\not\le R(U)$, there exists $y\in R(U)$ such that $xy\in R(W)$.
Note that $y\not\in G^2$, since otherwise $xy\in R(V)\cap R(W)\le R(U)$.
Now $[xy,x^2]=[y,x^2]\in W\cap U=(G')^4$, and then
\[
[y,x]^2 = [y,x^2] [y,x,x]^{-1} \in V,
\]
since $[y,x,x]\in [G^2,x]$.
By using that $[yx,y^2]=[x,y^2]$ and that $yx=(xy)^x\in R(W)$, we obtain in the same way that
$[y,x]^2\in U$.
Hence $[y,x]^2\in (G')^4$ and then, by (\ref{eqn square injective}), we get $[y,x]\in (G')^2$.

On the other hand, $x\not\in D$ yields
\[
G' = [x,G] = [x,\langle x \rangle R(U)] = [x,R(U)].
\]
Since $y\not\in G^2$ and $|\overline{R(U)}|=4$, we can write $R(U)=\langle y,z,G^2 \rangle$ for some $z$.
Now, since $[x,y]\in (G')^2$ and $[x,G^2]\le (G')^2$, it follows that $G'=\langle [x,z], (G')^2 \rangle$.
This implies that $G'$ is cyclic, which is a contradiction.
\end{proof}

We can now proceed to prove Theorem A for the prime $2$.

\begin{theorem}
Let $G$ be a finite $2$-group, and assume that $d(G')\le 2$.
Then $G'=K_x(G)$ for a suitable $x\in G$.
\end{theorem}

\begin{proof}
We may assume that $d(G')=2$.
We split the proof into two cases:

\vspace{5pt}

\noindent
\textit{Case 1: $C\ne G$.}

\vspace{5pt}

Let $T=\gamma_3(G)(G')^2$.
The condition $C\ne G$ implies that $\gamma_3(G)\not\le (G')^2$ and, as a consequence,
$(G')^2<T<G'$.
Thus $T\maxG G'$.
On the other hand, if $U\maxG G'$ then clearly $\gamma_3(G)(G')^2\le U$ and so $T=U$.
Hence if $N\trianglelefteq G$ and $N<G'$ then $N\le T$.
Also $D=D(T)$ is a subgroup of $G$ and $[D,G]\le T$.
On the other hand,
\[
[G,C'] \le [G,C,C] \le [G',C] \le (G')^2,
\]
while $[G,G']\not\le (G')^2$.
It follows that $C'<G'$, and consequently $C'\le T$.

By Lemma \ref{lemma D}, we have $G'=[x,G]$ for every $x\not\in D$.
We are going to show that Lemma \ref{lemma LN} can be applied either with $L=G'$ and $N=T$ or with $L=T$ and $N=(G')^2$, depending on the values of some commutator subgroups.
In the latter case, Lemma \ref{lemma index 2} will complete the proof.

Suppose first that $[G',C]\le T^2$.
Then we take $x\not\in C$, which by Lemma \ref{lemma D} implies $x\not\in D$.
Hence $G=\langle x \rangle C$ and $G'/T=\langle [x,y]T \rangle$ for some $y\in C$.
Also, observe that $T/(G')^2=\langle [x,[x,y]] (G')^2 \rangle$.
Now we have $[x,y,y]\in [G',C]\le T^2$ and $[x,[x,y],[x,y]]\in [G',G']\le (G')^4$, since $G'$ is powerful.
Thus we are done in this case.

Therefore we assume that $[G',C]\not\le T^2$ in the remainder.
Observe that $|(G')^2:T^2|\le |G':T|=2$ by Lemma \ref{lemma powerful}.
Since $[G',C]$ is contained in $(G')^2$ but not in $T^2$, we have $|(G')^2:T^2|=2$.
Also $|T^2:(G')^4|\le 2$.

If $[T,G]\not\le T^2$  then since $G=\langle G\smallsetminus C\rangle$ we can choose $x\not\in C$ and $t\in T$ such that $(G')^2/T^2 = \langle [x,t] T^2 \rangle$.
Since $[x,t,t]\in [(G')^2,G']\le (G')^8\le T^4$ and we can argue with the chief factor $T/(G')^2$ as in the case $[G',C]\le T^2$, the result follows also in this case.

Suppose finally that $[T,G]\le T^2$.
Since $[G',C]\not\le T^2$ and
\[
[G',D] \le [D,G,G] \le [T,G] \le T^2,
\]
there exist $x\in C\smallsetminus D$ and $g\in G'$ such that $(G')^2/T^2=\langle [x,g]T^2 \rangle$.
Then $[x,g,g]\in [(G')^2,G']\le T^4$.
On the other hand, there exists $y\in G$ such that $G'/T=\langle [x,y]T \rangle$.
Since $C'\le T$, we have $y\not\in C$ and then $T/(G')^2=\langle [x,y,y] (G')^2 \rangle=\langle [x,y^2] (G')^2 \rangle$.
Now
\[
[x,y^2,y^2] \in [T,G^2] \le [T,G]^2 [T,G,G] \le [T,G]^2 [T^2,G] \le (G')^4,
\]
which completes the proof in this case.

\vspace{8pt}

\noindent
\textit{Case 2: $C=G$.}

\vspace{5pt}

By Proposition \ref{proposition D cup R not G}, there exists $x\in G$ such that $G'=[x,G]$ and
$(G')^2=[x,G^2]$.
Since $C=G$ implies that the sections $G'/(G')^2$ and $(G')^2/(G')^4$ are central in $G$, it follows from Lemma \ref{lemma central} that
\[
G'/(G')^2=K_{x(G')^2}(G/(G')^2)
\]
and
\[
(G')^2/(G')^4=K_{x(G')^4}(G^2/(G')^4).
\]
On the other hand, by Lemma \ref{lemma C}, we have
\[
[x,G^2,G^2] \le [(G')^2,G^2] \le (G')^8.
\]
Hence we can apply Lemma \ref{lemma LN} with $L=(G')^2$ and any $N\maxG (G')^2$, getting
$(G')^2\subseteq K_x(G)$.
Now we are done by applying Lemma \ref{lemma index 2}.
\end{proof}

We conclude by proving Theorem B.

\begin{proof}[Proof of Theorem B]
By Theorem A, for every $N\trianglelefteq_{\mathrm{o}} G$, all elements of the derived subgroup $(G/N)'$ can be written as commutators from a single element.
Let $X_N=\{x\in G \mid (G/N)'=K_{xN}(G/N)\}$, which is closed in $G$, being a union of cosets of $N$.
Clearly, the family $\{ X_N \}_{N\trianglelefteq_{\mathrm{o}} G}$ has the finite intersection property and, since $G$ is compact, $\cap_{N\trianglelefteq_{\mathrm{o}} G} \, X_N \ne \varnothing$.
If $x$ belongs to this intersection, then $(G/N)'=K_{xN}(G/N)$ for all $N\trianglelefteq_{\mathrm{o}} G$, and then
\[
G'
=
\overline{G'}
=
\cap_{N\trianglelefteq_{\mathrm{o}} G} \ G' N
=
\cap_{N\trianglelefteq_{\mathrm{o}} G} \ K_x(G) N
=
\overline{K_x(G)}
=
K_x(G),
\]
as desired.
Observe that the last equality follows since $K_x(G)$ is the image of $G$ under the continuous map
$g\mapsto x^{-1}g^{-1}xg$.
Thus $K_x(G)$ is compact in the Hausdorff space $G$, and is consequently closed in $G$.
\end{proof}


\begin{thebibliography}{99}

\bibitem{ber}
Y.\ Berkovich,
\textit{Groups of prime power order, Volume 1\/},
Walter de Gruyter, 2008.

\bibitem{black}
N.\ Blackburn,
On prime-power groups in which the derived group has two generators,
\textit{Proc.\ Cambridge Philos.\ Soc.\/} \textbf{53} (1957), 19--27.


\bibitem{DDMS}
J.D.\ Dixon, M.P.F.\ du Sautoy, A.\ Mann, and D.\ Segal,
\textit{Analytic pro-$p$ groups\/}, 2nd edition,
Cambridge University Press, 1999.

\bibitem{gur}
R.M.\ Guralnick,
Commutators and commutator subgroups.
\textit{Advances in Math.\/}, \textbf{45} (1982), 319--330.

\bibitem{KM1}
L.-C.\ Kappe, R.F.\ Morse,
On commutators in $p$-groups,
\textit{J.\ Group Theory\/} \textbf{8} (2005), 415--429.

\bibitem{KM2}
L.-C.\ Kappe, R.F.\ Morse,
On commutators in groups,
in \textit{Groups St Andrews 2005, Volume 2}, pp. 531--558,
Cambridge University Press, 2007.

\bibitem{khu}
E.I.\ Khukhro,
\textit{$p$-Automorphisms of finite $p$-groups\/},
Cambridge University Press, 1998.

\bibitem{LOST}
M.W.\ Liebeck, E.A.\ O'Brien, A.\ Shalev, and P.H.\ Tiep,
The Ore conjecture,
\textit{J.\ European Math.\ Soc.\/} \textbf{12} (2010), 939--1008.

\bibitem{mac}
I.D.\ Macdonald,
On cyclic commutator subgroups.
\textit{J.\ London Math.\ Soc. (1)\/} \textbf{38} (1963), 419--422.

\bibitem{mac2}
I.D.\ Macdonald,
\textit{The theory of groups\/},
Clarendon Press, 1968.

\bibitem{rod}
D.M.\ Rodney,
On cyclic derived subgroups,
\textit{J.\ London Math.\ Soc. (2)\/} \textbf{8} (1974), 642--646.

\end{thebibliography}
\end{document}